\documentclass[12pt,a4paper]{amsart}
\usepackage{amssymb}
\usepackage{amsmath}

\usepackage{xcolor}

\newtheorem{theorem}{Theorem}[section]
\newtheorem{lemma}{Lemma}[section]
\newtheorem{corollary}{Corollary}[section]
\theoremstyle{remark}
\newtheorem{example}{Example}[section]

\theoremstyle{remark}
\newtheorem{remark}{Remark}[section]
\newtheorem{remarks}{Remarks}[section]

\begin{document}
\title[Relations between certain classes \dots]{Some relations between certain classes of analytic functions}

\author[P. Goswami]{Pranay Goswami}
\address{Department of Mathematics, Amity University Rajasthan, Jaipur-303002, India}
\email{pranaygoswami83@gmail.com}

\author[T. Bulboac\u{a}]{Teodor Bulboac\u{a}}
\address{Faculty of Mathematics and Computer Science, Babe\c{s}-Bolyai University, 400084 Cluj-Napoca, Romania}
\email{bulboaca@math.ubbcluj.ro}

\author[S. K. Bansal]{Sanjay K. Bansal}
\address{Bansal School of Engineering and Technology, Jaipur-303904, India}
\email{bansalindain@gmail.com}

\date{}

\begin{abstract}
In the present paper we introduce and studied two subclasses of multivalent functions denoted by $\mathcal{M}^{\lambda}_{p,n}(\gamma;\beta)$ and $\mathcal{N}^{\lambda}_{p,n}(\mu,\eta;\delta)$. Further, by giving specific values of the parameters of our main results, we will find some connection between these two classes, and moreover, several consequences of main results are also discussed.
\end{abstract}


\date{}

\maketitle
\section{Introduction and Preliminaries}
\setcounter{equation}{0}

Let $\mathcal{H}(\mathrm{U})$ denote the class of analytic functions in the unit disk $\mathrm{U}=\left\{z\in\mathbb{C}:|z|<1\right\}$, and let $\mathcal{A}(p,n)$ be the subclass of $\mathcal{H}(\mathrm{U})$ of the form
$$f(z)=z^{p}+\sum_{k=p+n}^{\infty}a_k z^{k},\;\left(n,p\in\mathbb{N}\right).$$

A function $f\in\mathcal{A}(p,n)$ is said to be {\it multivalent starlike functions of order $\alpha$} in $\mathrm{U}$, if it satisfies following inequality
$$\operatorname{Re}\frac{zf^{\prime}(z)}{f(z)}>\alpha,\;z\in\mathrm{U},
\;\left(0\leq\alpha<p,\;p\in\mathbb{N}\right),$$
and we denote this class by $\mathcal{S}_{p,n}^*(\alpha)$.

A function $f\in\mathcal{A}(p,n)$ is said to be {\it multivalent convex functions of order $\alpha$} in $\mathrm{U}$, if it satisfies following inequality
$$\operatorname{Re}\left[1+\frac{zf^{\prime\prime}(z)}{f^{\prime}(z)}\right]>\alpha,\;z\in\mathrm{U},
\;\left(0\leq\alpha<p,\;p\in\mathbb{N}\right),$$
and we denote this class by $\mathcal{C}_{p,n}(\alpha)$.

In the recent paper of Irmark and Raina \cite{irmak}, the authors introduced the subclass $\mathcal{T}_{\lambda}(p;\alpha)$ of $\mathcal{A}(p):=\mathcal{A}(p,1)$, consisting on the functions $f\in\mathcal{A}(p)$ that satisfies the inequality
$$\operatorname{Re}\dfrac{zf^\prime(z)+\lambda{z^2}f^{\prime\prime}(z)}{(1-\lambda)f(z)+
\lambda{z}f^{\prime}(z)}>\alpha,\;z\in\mathrm{U},\;\left(0\leq\alpha<p,\;p\in\mathbb{N}\right),$$
with $0\leq\lambda\leq1$.

Motivated by the class $\mathcal{T}_\lambda(p;\alpha)$, we introduce two new subclasses of $\mathcal{A}(p,n)$.

Thus, let $\mathcal{M}^{\lambda}_{p,n}(\gamma;\beta)$ be the class of functions $f\in\mathcal{A}(p,n)$ that satisfy the condition
\begin{eqnarray*}
&\operatorname{Re}\left[(1-\gamma)\dfrac{z\mathcal{F}_\lambda'(z)}{\mathcal{F}_\lambda(z)}+
\gamma\left(1+\dfrac{z\mathcal{F}_\lambda''(z)}{\mathcal{F}_\lambda'(z)}\right)\right]>\beta,
\;z\in\mathrm{U},\\
&\left(0\leq\lambda\leq1,\;0\leq\beta<p,\;\gamma\in\mathbb{R},\;p\in\mathbb{N}\right),
\end{eqnarray*}
and let $\mathcal{N}^{\lambda}_{p,n}(\mu,\eta;\delta)$ be the class of functions $f\in\mathcal{A}(p,n)$ that satisfy the conditions
\begin{eqnarray*}
&\mathcal{F}_\lambda(z)\mathcal{F}_\lambda'(z)\neq0,
\;z\in\dot{\mathrm{U}}:=\mathrm{U}\setminus\{0\},
\end{eqnarray*}
and
\begin{eqnarray*}
&\operatorname{Re}\left[\left(\dfrac{\mathcal{F}_\lambda(z)}{z^p}\right)^\mu
\left(\dfrac{\mathcal{F}_\lambda'(z)}{z^{p-1}}\right)^\eta\right]>\delta,\;z\in\mathrm{U},\\
&\left(0\leq\lambda\leq1,\;0\leq\delta<p^\eta(1+\lambda(p-1))^{\eta+\mu},\;\mu,\eta\in\mathbb{R},\;
p\in\mathbb{N}\right)
\end{eqnarray*}
where
\begin{equation}\label{defFlambda}
\mathcal{F}_\lambda(z)=(1-\lambda)f(z)+\lambda zf^\prime(z),
\end{equation}
and all the powers are the principal ones.

From above definitions, the following subclasses of the classes $\mathcal{A}(p,n)$ and $\mathcal{A}(n)=\mathcal{A}(1,n)$ emerge from the families of the functions $\mathcal{M}^{\lambda}_{p,n}(\gamma;\beta)$ and $\mathcal{N}^{\lambda}_{p,n}(\mu,\eta;\delta)$:
$$\begin{array}{l}
\mathcal{M}_{p,n}^0(0;\beta)=\mathcal{N}_{p,n}^0(-1,1;\beta)=\mathcal{S}_{p,n}^*(\beta),\;
(0\leq\beta<p);\\[1mm]
\mathcal{M}_{1,n}^0(0;\beta)=\mathcal{N}_{1,n}^0(-1,1;\beta)=\mathcal{S}_{1,n}^*(\beta)=:
\mathcal{S}_{n}^*(\beta),\;(0\leq\beta<1);\\[1mm]
\mathcal{M}_{p,n}^1(0;\beta)=\mathcal{C}_{p,n}(\beta),\;(0\leq\beta<p);\\[1mm]
\mathcal{M}_{1,n}^1(0;\beta)=\mathcal{C}_{1,n}(\beta)=:\mathcal{C}_{n}(\beta),\;(0\leq\beta<1);\\[1mm]
\mathcal{M}_{1,p}^\lambda(0;\beta)=\mathcal{T}_\lambda(p;\beta),\;(0\leq\beta<p);\\[1mm]
\mathcal{N}_{1,n}^0(1,\eta ;\beta)=:\mathcal{B}_n(\eta;\beta),\;(\eta\geq-1,\;0\leq\beta<1).
\end{array}$$

Note that $\mathcal{S}_{p,n}^*(\beta)$, $\mathcal{C}_{p,n}(\beta)$, $\mathcal{S}_{n}^*(\beta)$, $\mathcal{C}_{n}(\beta)$ and ${\mathcal{B}_n}(\eta ;\beta)$ are said to be the class of multivalent starlike functions of order $\beta$, multivalent convex functions of order $\beta$, univalent starlike functions of order $\beta$, univalent convex functions of order $\beta$, and a subclass of Bazilevi\'{c} functions, respectively. Also, we remark that the class ${\mathcal{T}_\lambda}(p;\beta)$ was studied by Irmak and Raina \cite{irmak}.

Let denote by $\mathcal{H}[a,n]$ the class
$$\mathcal{H}[a,n]=\left\{p\in\mathcal{H}(\mathrm{U}):p(z)=a+a_nz^n+\dots,\;z\in\mathrm{U}\right\}.$$

In this paper we will extend the results of Irmak et al. \cite{irmak1} for multivalent functions, by defining the differential operator $\mathcal{J}_\lambda^p(\mu,\eta):\mathcal{A}_{p,n}\rightarrow\mathcal{H}[(\mu+\eta),p+n]$,
$$\mathcal{J}_\lambda^p(\mu,\eta)f(z)=\mu\frac{z\mathcal{F}_\lambda^\prime(z)}{\mathcal{F}_\lambda(z)}+
\eta\left(1+\frac{z\mathcal{F}_\lambda^{\prime\prime}(z)}{\mathcal{F}_\lambda^\prime(z)}\right),$$
and further find its relationship with $\mathcal{N}^{\lambda}_{p,n}(\mu,\eta;\delta)$.

To prove these results, we will require following Lemmas due to Miller and Mocanu \cite[p. 33--35]{miller}:

\begin{lemma}\label{lemma1}
Let $\Omega\subset\mathbb{C}$ and suppose that the function $\psi:\mathbb{C}^2\times\mathrm{U}\rightarrow\mathbb{C}$ satisfies $\psi(Me^{i\theta},Ke^{i\theta};z)\notin\Omega$ for all $K\geq Mn$, $\theta\in\mathbb{R}$, and $z\in\mathrm{U}$. If $p\in\mathcal{H}[0,n]$ and $\psi(p(z),zp'(z))\in\Omega$ for all $z\in\mathrm{U}$, then $|p(z)|<M$ for all $z\in\mathrm{U}$.
\end{lemma}

\begin{lemma}\label{lemma2}
Let $\Omega\subset\mathbb{C}$ and suppose that the function $\psi:\mathbb{C}^2\times\mathrm{U}\rightarrow\mathbb{C}$ satisfies $\psi(ix,y;z)\notin\Omega$ for all $x\in\mathbb{R}$, $y\leq-n(1+x^2)/2$, and $z\in\mathrm{U}$. If $p\in\mathcal{H}[1,n]$ and $\psi(p(z),zp'(z))\in\Omega$ for all $z\in\mathrm{U}$, then $\operatorname{Re}p(z)>0$ for all $z\in\mathrm{U}$.
\end{lemma}
\section{Main Results}
\setcounter{equation}{0}

\begin{theorem}\label{thm1}
Let $f\in\mathcal{A}(p,n)$ such that $\mathcal{F}_\lambda(z)\mathcal{F}_\lambda^\prime(z)\neq0$ for all $z\in\dot{\mathrm{U}}$, where $\mathcal{F}_\lambda$ is given by \eqref{defFlambda}, and let $\eta,\mu\in\mathbb{R}$. If
\begin{equation}\label{assumpt1thm1}
\operatorname{Re}\mathcal{J}_\lambda^p(\mu,\eta)f(z)<p(\mu+\eta)+
\frac{nM}{M+p^\eta(1+\lambda(p-1))^{\eta+\mu}},\;z\in\mathrm{U},
\end{equation}
where $M\geq p^\eta(1+\lambda(p-1))^{\eta+\mu}$, then
\begin{equation}\label{conclthm1}
\left|\left(\frac{\mathcal{F}_\lambda(z)}{z^p}\right)^\mu
\left(\frac{\mathcal{F}_\lambda^\prime(z)}{z^{p-1}}\right)^\eta-
p^\eta(1+\lambda(p-1))^{\eta+\mu}\right|<M,\;z\in\mathrm{U}.
\end{equation}
(All the powers are the principal ones.)
\end{theorem}

\begin{proof}
Let define the function $h$ by
$$h(z)=\left(\frac{\mathcal{F}_\lambda(z)}{z^p}\right)^\mu
\left(\frac{\mathcal{F}_\lambda^\prime(z)}{z^{p-1}}\right)^\eta-
p^\eta(1+\lambda(p-1))^{\eta+\mu}.$$

From the assumptions $f\in\mathcal{A}(p,n)$ with $\mathcal{F}_\lambda^\prime(z)\mathcal{F}_\lambda(z)\neq0$ for all $z\in\dot{\mathrm{U}}$, we have that $h\in\mathcal{H}[0,n]$, and a simple computation shows that
\begin{equation}\label{eq1thm1}
\mathcal{J}_\lambda^p(\mu,\eta)f(z)=p(\mu+\eta)+\frac{zh'(z)}{h(z)+p^\eta(1+\lambda(p-1))^{\eta+\mu}}.
\end{equation}
Letting
$$\psi(r,s;z)=p(\mu+\eta)+\frac{s}{r+p^\eta(1+\lambda(p-1))^{\eta+\mu}}$$
and
$$\Omega=\left\{w\in\mathbb{C}:\operatorname{Re}w<p(\mu+\eta)+
\frac{nM}{M+p^\eta(1+\lambda(p-1))^{\eta+\mu}}\right\},$$
from \eqref{eq1thm1} combined with the assumption \eqref{assumpt1thm1} we obtain that $$\psi(h(z),zh'(z);z)=\mathcal{J}_\lambda^p(\mu,\eta)f(z)\in\Omega\quad\text{for all}\quad z\in\mathrm{U}.$$
Further, for any $\theta\in\mathbb{R}$ and $K\geq nM$, since $M\geq p^\eta(1+\lambda(p-1))^{\eta+\mu}$ we also have
\begin{eqnarray*}
&\displaystyle\operatorname{Re}\psi\left(Me^{i\theta},Ke^{i\theta};z\right)=p(\mu+\eta)+K\operatorname{Re}
\frac{1}{M+e^{-i\theta}p^\eta(1+\lambda(p-1))^{\eta+\mu}}\\
&\displaystyle\geq p(\mu+\eta)+\frac{nM}{M+p^\eta(1+\lambda(p-1))^{\eta+\mu}},\;z\in\mathrm{U},
\end{eqnarray*}
that is $\psi\left(Me^{i\theta},Ke^{i\theta};z\right)\not\in\Omega$ for all $K\geq Mn$, $\theta\in\mathbb{R}$, and $z\in\mathrm{U}$. Therefore, according to Lemma \ref{lemma1} we obtain $|h(z)|<M$ for all $z\in\mathrm{U}$, hence the conclusion \eqref{conclthm1} is proved.
\end{proof}

\begin{theorem} \label{thm2}
Let $f\in\mathcal{A}(p,n)$ such that $\mathcal{F}_\lambda(z)\mathcal{F}_\lambda^\prime(z)\neq0$ for all $z\in\dot{\mathrm{U}}$, where $\mathcal{F}_\lambda$ is given by \eqref{defFlambda}, and let $\eta,\mu\in\mathbb{R}$. If
\begin{equation}\label{assumpt1thm2}
\operatorname{Re}\mathcal{J}_\lambda^p(\mu,\eta)f(z)>k(\mu,\eta,\lambda;\delta),\;z\in\mathrm{U},
\end{equation}
where $\delta\in\left[0,p^\eta(1+\lambda(p-1))^{\eta+\mu}\right)$ and
$$k(\mu,\eta,\lambda;\delta)=\left\{
\begin{array}{l}
p(\mu+\eta)-\frac{n\delta}{2\left[p^\eta(1+\lambda(p-1))^{\eta+\mu}-\delta\right]},\;\text{if}\\
\hspace{51mm}\delta\in\left[0,\frac{p^\eta(1+\lambda(p-1))^{\eta+\mu}}{2}\right],\\
p(\mu+\eta)-\frac{n\left[p^\eta(1+\lambda(p-1))^{\eta+\mu}-\delta\right]}{2\delta},\;\text{if}\\
\hspace{18mm}
\delta\in\left[\frac{p^\eta(1+\lambda(p-1))^{\eta+\mu}}{2},p^\eta(1+\lambda(p-1))^{\eta+\mu}\right),
\end{array}\right.$$
then $f\in\mathcal{N}^{\lambda}_{p,n}(\mu,\eta;\delta)$.
\end{theorem}

\begin{proof}
Let define the function $h$ by
$$h(z)=\frac{1}{p^\eta(1+\lambda(p-1))^{\eta+\mu}-\delta}
\left[\left(\frac{\mathcal{F}_\lambda(z)}{z^p}\right)^\mu
\left(\frac{\mathcal{F}_\lambda^\prime(z)}{z^{p-1}}\right)^\eta-\delta\right],$$
where all the powers are the principal ones. From the assumptions $f\in\mathcal{A}(p,n)$ with $\mathcal{F}_\lambda^\prime(z)\mathcal{F}_\lambda(z)\neq0$ for all $z\in\dot{\mathrm{U}}$, we have that $h\in\mathcal{H}[1,n]$, and we may easily show that
\begin{equation}\label{eqn1thm2}
\mathcal{J}_\lambda^p(\mu,\eta)f(z)=p(\mu+\eta)+
\frac{\left[p^\eta\left(1+\lambda(p-1)\right)^{\eta+\mu}-\delta\right]zh'(z)}
{\left[p^\eta(1+\lambda(p-1))^{\eta+\mu}-\delta\right]h(z)+\delta}.
\end{equation}
Further, if we let
$$\psi(r,s;z)=p(\mu+\eta)+\frac{\left[p^\eta(1+\lambda(p-1))^{\eta+\mu}-\delta\right]s}
{\left[p^\eta(1+\lambda(p-1))^{\eta+\mu}-\delta\right]r+\delta}$$
and
$$\Omega=\left\{w\in\mathbb{C}:\operatorname{Re}w>k(\mu,\eta,\lambda;\delta)\right\},$$
from \eqref{eqn1thm2} combined with the assumption \eqref{assumpt1thm2} we get $$\psi(h(z),zh'(z);z)=\mathcal{J}_\lambda^p(\mu,\eta)f\in\Omega\quad\text{for all}\quad z\in\mathrm{U}.$$
Also, for any $x\in\mathbb{R}$, $y\leq-n(1+x^2)/2$, and $z\in\mathrm{U}$, we have
\begin{eqnarray*}
&\operatorname{Re}\psi(ix,y;z)=p(\mu+\eta)+\frac{\left[p^\eta(1+\lambda(p-1))^{\eta+\mu}-\delta\right]y}
{\delta^2+\left[p^\eta(1+\lambda(p-1))^{\eta+\mu}-\delta\right]^2x^2}\leq\\ &p(\mu+\eta)-\frac{n\delta\left[p^\eta(1+\lambda(p-1))^{\eta+\mu}-\delta\right]}{2}
\frac{x^2+1}{\delta^2+\left[p^\eta(1+\lambda(p-1))^{\eta+\mu}-\delta\right]^2x^2}=:H(x)\leq\\
&k(\mu,\eta,\lambda;\delta)=\left\{
\begin{array}{l}
\lim\limits_{x\to+\infty}H(x),\;\text{if}\;
\delta\in\left[0,\frac{p^\eta(1+\lambda(p-1))^{\eta+\mu}}{2}\right],\\
H(0),\;\text{if}\;\delta\in\left[\frac{p^\eta(1+\lambda(p-1))^{\eta+\mu}}{2},
p^\eta(1+\lambda(p-1))^{\eta+\mu}\right).
\end{array}\right.\\
\end{eqnarray*}
This shows that $\psi(ix,y;z)\not\in\Omega$ for all $x\in\mathbb{R}$, $y\leq-n(1+x^2)/2$, and $z\in\mathrm{U}$. From Lemma \ref{lemma2} it follows that $\operatorname{Re}h(z)>0$ for all $z\in\mathrm{U}$, which proves the conclusion of the theorem.
\end{proof}
\section{Corollaries and Consequences}
\setcounter{equation}{0}

In this section we will discuss some interesting consequences of our main theorems, that extend some previous results obtained in \cite{irmak1}.

Putting $\lambda=0$ in the Theorem \ref{thm1} and Theorem \ref{thm2} we get following corollaries:

\begin{corollary}\label{cor1}
Let $f\in\mathcal{A}(p,n)$ such that $f(z)f^\prime(z)\neq0$ for all $z\in\dot{\mathrm{U}}$, and let $\eta,\mu\in\mathbb{R}$. If
\begin{equation}\label{eqcor1a}
\operatorname{Re}\left[\mu\frac{zf'(z)}{f(z)}+\eta\left(1+\frac{zf''(z)}{f'(z)}\right)\right]<
p(\mu+\eta)+\frac{nM}{M+p^\eta},\;z\in\mathrm{U},
\end{equation}
where $M\geq p^\eta$, then
$$\left|\left(\frac{f(z)}{z^p}\right)^\mu\left(\frac{f^\prime(z)}{z^{p-1}}\right)^\eta-p^\eta\right|<
M,\;z\in\mathrm{U}.$$
(All the powers are the principal ones.)
\end{corollary}

\begin{example}\label{ex31}
Let $\mu,\eta\geq0$ with $\mu+\eta>0$, and let $a\in\mathbb{C}^*:=\mathbb{C}\setminus\{0\}$ such that
\begin{eqnarray}
&&|a|\leq\frac{p}{p+n},\;\left(n,p\in\mathbb{N}\right),\label{ex31a}\\
&&\mu\frac{|a|}{1+|a|}+\eta\frac{(p+n)|a|}{p+(p+n)|a|}\leq\frac{M}{M+p^\eta},\label{ex31b}
\end{eqnarray}
where $M\geq p^\eta$. Then,
$$\left|\left(1+az^n\right)^\mu\left[p+a\left(p+n\right)z^n\right]^\eta-p^\eta\right|<M,\;z\in\mathrm{U}.$$
\end{example}

\begin{proof}
Considering the function $f(z)=z^p+az^{p+n}$ with $a\in\mathbb{C}^*$, the assumption \eqref{ex31a} implies that $f(z)f^\prime(z)\neq0$ for all $z\in\dot{\mathrm{U}}$, while the condition \eqref{eqcor1a} reduces to
\begin{equation}\label{ineqex31aa}
\operatorname{Re}\left[\mu\frac{az^n}{1+az^n}+\eta\frac{a(p+n)z^n}{p+a(p+n)z^n}\right]<
\frac{M}{M+p^\eta},\;z\in\mathrm{U}.
\end{equation}

If we let $H(\zeta)=\dfrac{\zeta}{1+\zeta}$, with $|\zeta|\leq\rho\leq1$, it follows that
\begin{equation}\label{ineqReH}
-\frac{\rho}{1-\rho}\leq\operatorname{Re}H(\zeta)\leq\frac{\rho}{1+\rho},\;|\zeta|\leq\rho<1.
\end{equation}

According to this remark, we deduce that
$$\operatorname{Re}\left[\mu\frac{az^n}{1+az^n}+\eta\frac{a(p+n)z^n}{p+a(p+n)z^n}\right]<
\mu\frac{|a|}{1+|a|}+\eta\frac{(p+n)|a|}{p+(p+n)|a|},\;z\in\mathrm{U},$$
whenever $\mu,\eta\geq0$ and \eqref{ex31a} holds. Thus, the assumption \eqref{ex31b} implies that the inequality \eqref{ineqex31aa} holds, and from Corollary \ref{cor1} our result follows immediately.
\end{proof}

\begin{corollary}\label{cor2}
Let $f\in\mathcal{A}(p,n)$ such that $f(z)f^\prime(z)\neq0$ for all $z\in\dot{\mathrm{U}}$, and let $\eta,\mu\in\mathbb{R}$. If
$$\operatorname{Re}\left[\mu\frac{zf'(z)}{f(z)}+\eta\left(1+\frac{zf''(z)}{f'(z)}\right)\right]>
\nu(\eta,\mu;\delta),\;z\in\mathrm{U},$$
where $\delta\in\left[0,p^\eta\right)$ and
$$\nu(\eta,\mu;\delta):=k(\mu,\eta,0;\delta)=\left\{
\begin{array}{l}
p(\mu+\eta)-\frac{n\delta}{2\left(p^\eta-\delta\right)},\quad\text{if}\quad
\delta\in\left[0,\frac{p^\eta}{2}\right],\\[2mm]
p(\mu+\eta)-\frac{n\left(p^\eta-\delta\right)}{2\delta},\quad\text{if}\quad
\delta\in\left[\frac{p^\eta}{2},p^\eta\right),
\end{array}\right.$$
then
$$\operatorname{Re}\left[\left(\frac{f(z)}{z^p}\right)^\mu
\left(\frac{f'(z)}{z^{p-1}}\right)^\eta\right]>\delta,\;z\in\mathrm{U}.$$
(All the powers are the principal ones.)
\end{corollary}

Using a similar proof like to Example \ref{ex31}, the following result can be easily deduced from the above corollary:

\begin{example}\label{ex32}
Let $\mu,\eta\leq0$ with $\mu+\eta<0$, and let $a\in\mathbb{C}^*$ such that
\begin{eqnarray*}
&&|a|\leq\frac{p}{p+n},\;\left(n,p\in\mathbb{N}\right),\\
&&\mu\frac{|a|}{1+|a|}+\eta\frac{(p+n)|a|}{p+(p+n)|a|}\geq\frac{\delta}{2\left(\delta-p^\eta\right)},
\;\;if\;\;\delta\in\left[0,\frac{p^\eta}{2}\right],\\
&&\mu\frac{|a|}{1+|a|}+\eta\frac{(p+n)|a|}{p+(p+n)|a|}\geq\frac{\delta-p^\eta}{2\delta},
\;\;if\;\;\delta\in\left[\frac{p^\eta}{2},p^\eta\right),
\end{eqnarray*}
where $\delta\in\left[0,p^\eta\right)$. Then,
$$\operatorname{Re}\left\{\left(1+az^n\right)^\mu\left[p+a\left(p+n\right)z^n\right]^\eta\right\}>\delta,
\;z\in\mathrm{U}.$$
\end{example}

\begin{remark}\label{rem1}
Taking $p=1$ in the Corollaries \ref{cor1} and \ref{cor2}, we get a known result obtain by Irmak et al. \cite{irmak1}.
\end{remark}

Taking $\lambda=0$, $\mu=1-\gamma$, and $\eta=\gamma$ in Theorem \ref{thm1} and \ref{thm2}, respectively, we obtain the following special cases:

\begin{corollary}\label{cor3}
Let $f\in\mathcal{A}(p,n)$ such that $f(z)f^\prime(z)\neq0$ for all $z\in\dot{\mathrm{U}}$, and let $\gamma\in\mathbb{R}$. If
\begin{equation}\label{eqasscor3}
\operatorname{Re}\left[(1-\gamma)\frac{zf'(z)}{f(z)}+\gamma\left(1+\frac{zf''(z)}{f'(z)}\right)\right]<
p+\frac{nM}{M+p^\gamma},\;z\in\mathrm{U},
\end{equation}
where $M\geq p^\gamma$, then
$$\left|\left(\frac{f(z)}{z^p}\right)^{1-\gamma}\left(\frac{f^\prime(z)}{z^{p-1}}\right)^\gamma-
p^\gamma\right|<M,\;z\in\mathrm{U}.$$
(All the powers are the principal ones.)
\end{corollary}

\begin{example}\label{ex33}
Let $\gamma\geq0$, and let $a\in\mathbb{C}^*$ such that
\begin{eqnarray}
&&|a|\leq p,\;\left(p\in\mathbb{N}\right),\label{ex33a}\\
&&|a|+\gamma\frac{|a|}{p+|a|}\leq\frac{M}{M+p^\gamma},\label{ex33b}
\end{eqnarray}
where $M\geq p^\gamma$. Then,
$$\left|e^{az}\left(p+az\right)^\gamma-p^\gamma\right|<M,\;z\in\mathrm{U}.$$
\end{example}

\begin{proof}
For the function $f(z)=z^pe^{az}\in\mathcal{A}(p)$ with $a\in\mathbb{C}^*$, the assumption \eqref{ex33a} implies that $f(z)f^\prime(z)\neq0$ for all $z\in\dot{\mathrm{U}}$, and the condition \eqref{eqasscor3} reduces to
\begin{equation}\label{ineqex33aa}
\operatorname{Re}\left(az+\gamma\frac{az}{p+az}\right)<\frac{M}{M+p^\gamma},\;z\in\mathrm{U}.
\end{equation}

Since $\operatorname{Re}(az)<|a|$ for $z\in\mathrm{U}$, using the inequality \eqref{ineqReH} we deduce that
$$\operatorname{Re}\left(az+\gamma\frac{az}{p+az}\right)<
|a|+\gamma\frac{|a|}{p+|a|},\;z\in\mathrm{U},$$
whenever $\gamma\geq0$ and the condition \eqref{ex33a} is satisfied. The assumption \eqref{ex33b} implies that the inequality \eqref{ineqex33aa} holds, and from Corollary \ref{cor3} we obtain our result.
\end{proof}

\begin{corollary}\label{cor4}
Let $f\in\mathcal{A}(p,n)$ such that $f(z)f^\prime(z)\neq0$ for all $z\in\dot{\mathrm{U}}$, and let $\gamma\in\mathbb{R}$. If
$$\operatorname{Re}\left[(1-\gamma)\frac{zf'(z)}{f(z)}+\gamma\left(1+\frac{zf''(z)}{f'(z)}\right)\right]>
\xi(\gamma;\delta),\;z\in\mathrm{U},$$
where $\delta\in\left[0,p^\gamma\right)$ and
$$\xi(\gamma;\delta):=k(1-\gamma,\gamma,0;\delta)=\left\{
\begin{array}{l}
p-\frac{n\delta}{2\left(p^\gamma-\delta\right)},\quad\text{if}\quad
\delta\in\left[0,\frac{p^\gamma}{2}\right],\\[2mm]
p-\frac{n\left(p^\gamma-\delta\right)}{2\delta},\quad\text{if}\quad
\delta\in\left[\frac{p^\gamma}{2},p^\eta\right),
\end{array}\right.$$
then
$$\operatorname{Re}\left[\left(\frac{f(z)}{z^p}\right)^{1-\gamma}
\left(\frac{f'(z)}{z^{p-1}}\right)^\gamma\right]>\delta,\;z\in\mathrm{U}.$$
(All the powers are the principal ones.)
\end{corollary}

Like in the proof of Example \ref{ex33}, from the above corollary we obtain the next special case:

\begin{example}\label{ex34}
Let $\gamma\geq0$, and let $a\in\mathbb{C}^*$ such that
\begin{eqnarray*}
&&|a|\leq p,\;\left(p\in\mathbb{N}\right),\\
&&-|a|+\gamma\frac{|a|}{p+|a|}\geq\frac{\delta}{2\left(\delta-p^\gamma\right)},
\;\;if\;\;\delta\in\left[0,\frac{p^\gamma}{2}\right],\\
&&-|a|+\gamma\frac{|a|}{p+|a|}\geq\frac{\delta-p^\gamma}{2\delta},
\;\;if\;\;\delta\in\left[\frac{p^\gamma}{2},p^\gamma\right),
\end{eqnarray*}
where $\delta\in\left[0,p^\gamma\right)$. Then,
$$\operatorname{Re}\left\{e^{az}(p+az)^\gamma\right\}>\delta,\;z\in\mathrm{U}.$$
\end{example}

\begin{remarks}
1. For $p=1$ and $M=\gamma+1$, Corollary \ref{cor3} reduces to a known result due to Irmak et al. \cite{irmak1}.

2. For $p=1$, Corollary \ref{cor4} reduces to a known result due to Irmak et al. \cite{irmak1}.
\end{remarks}

Upon taking $\lambda=1$ in Theorem \ref{thm1} and Theorem \ref{thm2} we get the next special cases:

\begin{corollary}\label{cor5}
Let $f\in\mathcal{A}(p,n)$ such that $f^\prime(z)\left(f^\prime(z)+zf^{\prime\prime}(z)\right)\neq0$ for all $z\in\dot{\mathrm{U}}$, and let $\eta,\mu\in\mathbb{R}$. If
$$\operatorname{Re}\mathcal{J}_1^p(\mu,\eta)f(z)<p(\mu+\eta)+\frac{nM}{M+p^{2\eta+\mu}},\;z\in\mathrm{U},$$
where $M\geq p^{2\eta+\mu}$, then
$$\left|\left(\frac{f^\prime(z)}{z^{p-1}}\right)^\mu
\left(\frac{f^\prime(z)+zf^{\prime\prime}(z)}{z^{p-1}}\right)^\eta-p^{2\eta+\mu}\right|<M,
\;z\in\mathrm{U}.$$
(All the powers are the principal ones.)
\end{corollary}

\begin{corollary}\label{cor6}
Let $f\in\mathcal{A}(p,n)$ such that $f^\prime(z)\left(f^\prime(z)+zf^{\prime\prime}(z)\right)\neq0$ for all $z\in\dot{\mathrm{U}}$, and let $\eta,\mu\in\mathbb{R}$. If
$$\operatorname{Re}\mathcal{J}_1^p(\mu,\eta)f(z)>\sigma(\mu,\eta;\delta),\;z\in\mathrm{U},$$
where $\delta\in\left[0,p^{2\eta+\mu}\right)$ and
$$\sigma(\mu,\eta;\delta):=k(\mu,\eta,1;\delta)=\left\{
\begin{array}{l}
p(\mu+\eta)-\frac{n\delta}{2\left(p^{2\eta+\mu}-\delta\right)},\quad\text{if}\quad
\delta\in\left[0,\frac{p^{2\eta+\mu}}{2}\right],\\
p(\mu+\eta)-\frac{n\left(p^{2\eta+\mu}-\delta\right)}{2\delta},\quad\text{if}\quad
\delta\in\left[\frac{p^{2\eta+\mu}}{2},p^{2\eta+\mu}\right),
\end{array}\right.$$
then
$$\operatorname{Re}\left[\left(\frac{f^\prime(z)}{z^{p-1}}\right)^\mu
\left(\frac{f^\prime(z)+zf^{\prime\prime}(z)}{z^{p-1}}\right)^\eta\right]>\delta,\;z\in\mathrm{U}.$$
(All the powers are the principal ones.)
\end{corollary}

If we let $f(z)=z^p+az^{p+n}$ with $a\in\mathbb{C}^*$, we easily deduce that $f^\prime(z)\left(f^\prime(z)+zf^{\prime\prime}(z)\right)\neq0$ for all $z\in\dot{\mathrm{U}}$ whenever $|a|\leq\dfrac{p^2}{(p+n)^2}$. Using the fact that $\mathcal{F}_1(z)=zf'(z)$, after some simple computations, from the above two corollaries we get respectively:

\begin{example}\label{ex35}
Let $\mu,\eta\in\mathbb{R}$, and let $a\in\mathbb{C}^*$ with $|a|\leq\dfrac{p^2}{(p+n)^2}$, $n\in\mathbb{N}$, such that
$$\operatorname{Re}\left[(\mu+\eta)\varphi(z)+\frac{z\varphi'(z)}{n\varphi(z)+p}\right]<
\frac{M}{M+p^{2\eta+\mu}},\;z\in\mathrm{U},$$
where $M\geq p^{2\eta+\mu}$ and $\varphi(z)=\dfrac{a(p+n)z^n}{p+(p+n)z^n}$. Then,
$$\left|\left[p+a(p+n)z^n\right]^\mu\left[p^2+a(p+n)^2z^n\right]^\eta-p^{2\eta+\mu}\right|<M,
\;z\in\mathrm{U}.$$
\end{example}

\begin{example}\label{ex36}
Let $\mu,\eta\in\mathbb{R}$, and let $a\in\mathbb{C}^*$ with $|a|\leq\dfrac{p^2}{(p+n)^2}$, $n\in\mathbb{N}$, such that
\begin{align*}
&\operatorname{Re}\left[(\mu+\eta)\varphi(z)+\dfrac{z\varphi'(z)}{n\varphi(z)+p}\right]>
\dfrac{\delta}{2\left(\delta-p^{2\eta+\mu}\right)},\;z\in\mathrm{U},
\;\;if\;\;\delta\in\left[0,\frac{p^{2\eta+\mu}}{2}\right],\\
&\operatorname{Re}\left[(\mu+\eta)\varphi(z)+\dfrac{z\varphi'(z)}{n\varphi(z)+p}\right]>
\dfrac{\delta-p^{2\eta+\mu}}{2\delta},\;z\in\mathrm{U},
\;\;if\;\;\delta\in\left[\frac{p^{2\eta+\mu}}{2},p^{2\eta+\mu}\right),
\end{align*}
where $\delta\in\left[0,p^{2\eta+\mu}\right)$ and $\varphi(z)=\dfrac{a(p+n)z^n}{p+(p+n)z^n}$. Then,
$$\operatorname{Re}\left\{\left[p+a(p+n)z^n\right]^\mu\left[p^2+a(p+n)^2z^n\right]^\eta\right\}>\delta,
\;z\in\mathrm{U}.$$
\end{example}

Taking $\lambda=1$, $\mu=1-\gamma$, and $\eta=\gamma$ in Theorem \ref{thm1} and \ref{thm2}, respectively, we obtain following results:

\begin{corollary}\label{cor7}
Let $f\in\mathcal{A}(p,n)$ such that $f^\prime(z)\left(f^\prime(z)+zf^{\prime\prime}(z)\right)\neq0$ for all $z\in\dot{\mathrm{U}}$, and let $\gamma\in\mathbb{R}$. If
$$\operatorname{Re}\mathcal{J}_1^p(1-\gamma,\gamma)f(z)<p+\frac{nM}{M+p^{\gamma+1}},\;z\in\mathrm{U},$$
where $M\geq p^{\gamma+1}$, then
$$\left|\left(\frac{f^\prime(z)}{z^{p-1}}\right)^{1-\gamma}
\left(\frac{f^\prime(z)+zf^{\prime\prime}(z)}{z^{p-1}}\right)^\gamma-p^{\gamma+1}\right|<M,
\;z\in\mathrm{U}.$$
(All the powers are the principal ones.)
\end{corollary}

\begin{corollary}\label{cor8}
Let $f\in\mathcal{A}(p,n)$ such that $f^\prime(z)\left(f^\prime(z)+zf^{\prime\prime}(z)\right)\neq0$ for all $z\in\dot{\mathrm{U}}$, and let $\gamma\in\mathbb{R}$. If
$$\operatorname{Re}\mathcal{J}_1^p(1-\gamma,\gamma)f(z)>\varrho(\gamma;\delta),
\;z\in\mathrm{U},$$
where $\delta\in\left[0,p^{\gamma+1}\right)$ and
$$\varrho(\gamma;\delta):=k(1-\gamma,\gamma,1;\delta)=\left\{
\begin{array}{l}
p-\frac{n\delta}{2\left(p^{\gamma+1}-\delta\right)},\quad\text{if}\quad
\delta\in\left[0,\frac{p^{\gamma+1}}{2}\right],\\
p-\frac{n\left(p^{\gamma+1}-\delta\right)}{2\delta},\quad\text{if}\quad
\delta\in\left[\frac{p^{\gamma+1}}{2},p^{\gamma+1}\right),
\end{array}\right.$$
$$\operatorname{Re}\left[\left(\frac{f^\prime(z)}{z^{p-1}}\right)^{1-\gamma}
\left(\frac{f^\prime(z)+zf^{\prime\prime}(z)}{z^{p-1}}\right)^\gamma\right]>\delta,\;z\in\mathrm{U}.$$
(All the powers are the principal ones.)
\end{corollary}

We will give two special cases of the above last corollaries, obtained for $f(z)=z^pe^{az}\in\mathcal{A}(p)$ with $a\in\mathbb{R}^*=\mathbb{R}\setminus\{0\}$. Since $f'(z)=z^{p-1}e^{az}(p+az)$, then $f^\prime(z)\neq0$ for all $z\in\dot{\mathrm{U}}$ if and only if $|a|\leq p$. Also, it could be easily checked that $f^\prime(z)+zf^{\prime\prime}(z)=z^{p-1}e^{az}\left[p^2+a(2p+1)z+a^2z^2\right]\neq0$ for all $z\in\dot{\mathrm{U}}$ whenever $|a|\leq\dfrac{2p+1-\sqrt{4p+1}}{2}\leq p$. Using the fact that $\mathcal{F}_1(z)=zf'(z)=z^{p}e^{az}(p+az)$, after some simple computations, we get respectively:

\begin{example}\label{ex37}
Let $\gamma\in\mathbb{R}$, and let $a\in\mathbb{R}^*$ with $|a|\leq\dfrac{2p+1-\sqrt{4p+1}}{2}$, $p\in\mathbb{N}$, such that
$$\operatorname{Re}\left\{az\left[1+\frac{1-\gamma}{az+p}+
\gamma\frac{2az+2p+1}{a^2z^2+a(2p+1)z+p^2}\right]\right\}<\frac{M}{M+p^{\gamma+1}},\;z\in\mathrm{U},$$
where $M\geq p^{\gamma+1}$. Then,
$$\left|e^{az}\left(p+az\right)^{1-\gamma}\left[a^2z^2+a(2p+1)z+p^2\right]^\gamma-p^{\gamma+1}\right|<M,
\;z\in\mathrm{U}.$$
\end{example}

\begin{example}\label{ex38}
Let $\gamma\in\mathbb{R}$, and let $a\in\mathbb{R}^*$ with $|a|\leq\dfrac{2p+1-\sqrt{4p+1}}{2}$, $p\in\mathbb{N}$, such that
\begin{align*}
&\operatorname{Re}\left\{az\left[1+\frac{1-\gamma}{az+p}+
\gamma\frac{2az+2p+1}{a^2z^2+a(2p+1)z+p^2}\right]\right\}>
\dfrac{\delta}{2\left(\delta-p^{\gamma+1}\right)},\;z\in\mathrm{U},\\
&\hspace{100mm}if\;\;\delta\in\left[0,\frac{p^{\gamma+1}}{2}\right],\\
&\operatorname{Re}\left\{az\left[1+\frac{1-\gamma}{az+p}+
\gamma\frac{2az+2p+1}{a^2z^2+a(2p+1)z+p^2}\right]\right\}>
\dfrac{\delta-p^{\gamma+1}}{2\delta},\;z\in\mathrm{U},\\
&\hspace{95mm}if\;\;\delta\in\left[\frac{p^{\gamma+1}}{2},p^{\gamma+1}\right),
\end{align*}
where $\delta\in\left[0,p^{\gamma+1}\right)$. Then,
$$\operatorname{Re}\left\{e^{az}\left(p+az\right)^{1-\gamma}
\left[a^2z^2+a(2p+1)z+p^2\right]^\gamma\right\}>\delta,\;z\in\mathrm{U}.$$
\end{example}

In the next result we will find the relation between $\mathcal{M}^{\lambda}_{p,n}(\gamma;\delta)$ and $\mathcal{N}^{\lambda}_{p,n}({1-\gamma,\gamma;\rho})$. For this purpose, putting $\mu=1-\gamma$, $\eta=\gamma$ in Theorem \ref{thm2}, we have:

\begin{corollary} \label{cor9}
Let $f\in\mathcal{A}(p,n)$ such that $\mathcal{F}_\lambda(z)\mathcal{F}_\lambda^\prime(z)\neq0$ for all $z\in\dot{\mathrm{U}}$, where $\mathcal{F}_\lambda$ is given by \eqref{defFlambda}, and let $\gamma\in\mathbb{R}$. If
$$f\in\mathcal{M}^{\lambda}_{p,n}(\gamma;\rho(\gamma,\lambda;\delta)),$$
where
$$\rho(\gamma,\lambda;\delta):=k(1-\gamma,\gamma,\lambda;\delta)=\left\{
\begin{array}{l}
p-\frac{n\delta}{2\left[p^\gamma(1+\lambda(p-1))-\delta\right]},\quad\text{if}\quad
\delta\in\left[0,\frac{p^\gamma(1+\lambda(p-1))}{2}\right],\\[2mm]
p-\frac{n\left[p^\gamma(1+\lambda(p-1))-\delta\right]}{2\delta},\quad\text{if}\\
\hspace{18mm}
\delta\in\left[\frac{p^\gamma(1+\lambda(p-1))}{2},p^\gamma(1+\lambda(p-1))\right),
\end{array}\right.$$
and $\delta\in\left[0,p^\gamma(1+\lambda(p-1))\right)$, then $f\in\mathcal{N}^{\lambda}_{p,n}(1-\gamma,\gamma;\delta)$.
\end{corollary}

For $\lambda=1/2$ we will give a special case of this last corollary, considering $f(z)=z^p+az^{p+n}$ with $a\in\mathbb{C}^*$. Since
$$\mathcal{F}_{1/2}(z)=\dfrac{1}{2}\left[f(z)+zf'(z)\right]=\frac{z^p}{2}\left[p+1+a(p+n+1)z^n\right],$$
it follows that $\mathcal{F}_{1/2}(z)\mathcal{F}_{1/2}^\prime(z)\neq0$ for all $z\in\dot{\mathrm{U}}$ whenever $|a|\leq\dfrac{p(p+1)}{(p+n)(p+n+1)}$. It is easy to check that $f\in\mathcal{M}^{1/2}_{p,n}(\gamma;\rho(\gamma,1/2;\delta))$ if and only if
$$\operatorname{Re}\left[p+\varphi(z)+\gamma\frac{z\varphi'(z)}{p+\varphi(z)}\right]>
\rho(\gamma,1/2;\delta),\;z\in\mathrm{U},$$
where $\varphi(z)=\dfrac{an(p+n+1)z^n}{p+1+a(p+n+1)z^n}$ and $\rho(\gamma,1/2;\delta)$ is given by Corollary \ref{cor9}. Because $\varphi(0)=0$, the above inequality holds in a neighborhood of the annulus, hence it holds for all $z\in\mathrm{U}$ if $|a|$ is small enough. Thus, we get the following example:

\begin{example}\label{ex39}
Let $\gamma\in\mathbb{R}$, and let $a\in\mathbb{C}^*$ with $|a|\leq\dfrac{p(p+1)}{(p+n)(p+n+1)}$, $n,p\in\mathbb{N}$, and $|a|$ small enough such that
$$f(z)=z^p+az^{p+n}\in\mathcal{M}^{1/2}_{p,n}(\gamma;\rho(\gamma,1/2;\delta)),$$
where $\rho(\gamma,1/2;\delta)$ is given by Corollary \ref{cor9}. Then,
$$f(z)=z^p+az^{p+n}\in\mathcal{N}^{1/2}_{p,n}(1-\gamma,\gamma;\delta),$$
i.e.
$$\operatorname{Re}\left\{\left[p+1+a(p+n+1)z^n\right]^{1-\gamma}
\left[p(p+1)+a(p+n)(p+n+1)z^n\right]^\gamma\right\}>2\delta$$
for all $z\in\mathrm{U}$.
\end{example}

Taking $\gamma=0$ and $n=1$ in above corollary, we get the next special case:

\begin{corollary} \label{cor10}
Let $f\in\mathcal{A}(p)$ such that $\mathcal{F}_\lambda(z)\mathcal{F}_\lambda^\prime(z)\neq0$ for all $z\in\dot{\mathrm{U}}$, where $\mathcal{F}_\lambda$ is given by \eqref{defFlambda}. If
$$f\in\mathcal{T}_{\lambda}(p;\rho_1(\lambda;\delta)),$$
where
$$\rho_1(\lambda;\delta):=k(1,0,\lambda;\delta)=\left\{
\begin{array}{l}
p-\frac{\delta}{2\left[1+\lambda(p-1)-\delta\right]},\quad\text{if}\quad
\delta\in\left[0,\frac{1+\lambda(p-1)}{2}\right],\\[2mm]
p-\frac{\left[1+\lambda(p-1)-\delta\right]}{2\delta},\quad\text{if}\\
\hspace{21mm}
\delta\in\left[\frac{1+\lambda(p-1)}{2},1+\lambda(p-1)\right),
\end{array}\right.$$
and $\delta\in\left[0,1+\lambda(p-1)\right)$, then $f\in\mathcal{N}^{\lambda}_{p,1}(1,0;\delta)$.
\end{corollary}

If $\lambda=1/2$ we will give a special case of the above result for $f(z)=z^p+az^{p+1}\in\mathcal{A}(p)$ with $a\in\mathbb{C}^*$. Because
$$\mathcal{F}_{1/2}(z)=\dfrac{1}{2}\left[f(z)+zf'(z)\right]=\frac{z^p}{2}\left[p+1+a(p+2)z\right],$$
it follows that $\mathcal{F}_{1/2}(z)\mathcal{F}_{1/2}^\prime(z)\neq0$ for all $z\in\dot{\mathrm{U}}$, whenever $|a|\leq\dfrac{p}{p+2}$. It is easy to check that $f\in\mathcal{T}_{1/2}(p;\rho_1(1/2;\delta))$ if and only if
\begin{equation}\label{eqCor10pr1}
\operatorname{Re}\left[p+\varphi(z)\right]>\rho_1(1/2;\delta),
\;z\in\mathrm{U},
\end{equation}
where $\varphi(z)=\dfrac{a(p+2)z}{p+1+a(p+2)z}$ and $\rho_1(1/2;\delta)$ is given by Corollary \ref{cor10}. According to \eqref{ineqReH} we deduce that
$$\operatorname{Re}\left[p+\varphi(z)\right]>p-\frac{|a|(p+2)}{p+1-|a|(p+2)},\;z\in\mathrm{U},$$
and combining this inequality with \eqref{eqCor10pr1} we obtain the following example:

\begin{example}\label{ex310}
Let $a\in\mathbb{C}^*$ with $|a|\leq\dfrac{p}{p+2}$, $p\in\mathbb{N}$, such that
\begin{align*}
&|a|\leq\frac{(p+1)\delta}{(p+2)(p+1-\delta)},\;if\;\delta\in\left[0,\frac{p+1}{4}\right],\\
&|a|\leq\frac{(p+1)(p+1-2\delta)}{(p+2)(p+1+2\delta)},\;if\;
\delta\in\left[\frac{p+1}{4},\frac{p+1}{2}\right),
\end{align*}
where $\delta\in\left[0,\dfrac{p+1}{2}\right)$. Then,
$$f(z)=z^p+az^{p+1}\in\mathcal{N}^{1/2}_{p,1}(1,0;\delta).$$
\end{example}

For the special cases $\eta=1$ and $\mu=-1$, Theorem \ref{thm1} and \ref{thm2} reduces, respectively, to the next results:

\begin{corollary} \label{cor11}
Let $f\in\mathcal{A}(p,n)$ such that $\mathcal{F}_\lambda(z)\mathcal{F}_\lambda^\prime(z)\neq0$ for all $z\in\dot{\mathrm{U}}$, where $\mathcal{F}_\lambda$ is given by \eqref{defFlambda}. If
$$\operatorname{Re}\mathcal{J}_\lambda^p(-1,1)f(z)<\frac{nM}{M+p},\;z\in\mathrm{U},$$
where $M\geq p$, then
$$\left|\frac{z\mathcal{F}_\lambda^\prime(z)}{\mathcal{F}_\lambda(z)}-p\right|<M,\;z\in\mathrm{U}.$$
\end{corollary}

\begin{corollary} \label{cor12}
Let $f\in\mathcal{A}(p,n)$ such that $\mathcal{F}_\lambda(z)\mathcal{F}_\lambda^\prime(z)\neq0$ for all $z\in\dot{\mathrm{U}}$, where $\mathcal{F}_\lambda$ is given by \eqref{defFlambda}. If
$$\operatorname{Re}\mathcal{J}_\lambda^p(-1,1)f(z)>\varsigma(\delta),\;z\in\mathrm{U},$$
where
$$\varsigma(\delta):=k(-1,1,\lambda;\delta)=\left\{
\begin{array}{l}
-\frac{n\delta}{2\left(p-\delta\right)},\quad\text{if}\quad
\delta\in\left[0,\frac{p}{2}\right],\\
-\frac{n\left(p-\delta\right)}{2\delta},\quad\text{if}\quad
\delta\in\left[\frac{p}{2},p\right),
\end{array}\right.$$
and $\delta\in\left[0,p\right)$, then
$$\operatorname{Re}\frac{z\mathcal{F}_\lambda^\prime(z)}{\mathcal{F}_\lambda(z)}>\delta,
\;z\in\mathrm{U}.$$
\end{corollary}

Reversely, taking $\eta=-1$ and $\mu=1$ in Theorem \ref{thm1} and \ref{thm2}, respectively, we get following results:

\begin{corollary} \label{cor13}
Let $f\in\mathcal{A}(p,n)$ such that $\mathcal{F}_\lambda(z)\mathcal{F}_\lambda^\prime(z)\neq0$ for all $z\in\dot{\mathrm{U}}$, where $\mathcal{F}_\lambda$ is given by \eqref{defFlambda}. If
$$\operatorname{Re}\mathcal{J}_\lambda^p(1,-1)f(z)<\frac{nMp}{Mp+1},\;z\in\mathrm{U},$$
where $M\geq\dfrac{1}{p}$, then
$$\left|\frac{\mathcal{F}_\lambda(z)}{z\mathcal{F}_\lambda^\prime(z)}-\frac{1}{p}\right|<M,
\;z\in\mathrm{U}.$$
\end{corollary}

\begin{corollary} \label{cor14}
Let $f\in\mathcal{A}(p,n)$ such that $\mathcal{F}_\lambda(z)\mathcal{F}_\lambda^\prime(z)\neq0$ for all $z\in\dot{\mathrm{U}}$, where $\mathcal{F}_\lambda$ is given by \eqref{defFlambda}. If
$$\operatorname{Re}\mathcal{J}_\lambda^p(1,-1)f(z)>\varsigma_1(1,-1,\lambda;\delta),\;z\in\mathrm{U},$$
where
$$\varsigma_1(1,-1,\lambda;\delta):=k(1,-1,\lambda;\delta)=\left\{
\begin{array}{l}
-\frac{n\delta p}{2\left(1-\delta p\right)},\quad\text{if}\quad
\delta\in\left[0,\frac{1}{2p}\right],\\[2mm]
-\frac{n\left(1-\delta p\right)}{2\delta p},\quad\text{if}\quad
\delta\in\left[\frac{1}{2p},\frac{1}{p}\right),
\end{array}\right.$$
and $\delta\in\left[0,\dfrac{1}{p}\right)$, then
$$\operatorname{Re}\frac{\mathcal{F}_\lambda(z)}{z\mathcal{F}_\lambda^\prime(z)}>\delta,
\;z\in\mathrm{U}.$$
\end{corollary}

For $\lambda=1/2$ we will give some special cases of the last five results, considering $f(z)=z^p+az^{p+1}\in\mathcal{A}(p)$ with $a\in\mathbb{C}^*$. Because
$$\mathcal{F}_{1/2}(z)=\dfrac{1}{2}\left[f(z)+zf'(z)\right]=\frac{z^p}{2}\left[p+1+a(p+2)z\right],$$
it follows that $\mathcal{F}_{1/2}(z)\mathcal{F}_{1/2}^\prime(z)\neq0$ for all $z\in\dot{\mathrm{U}}$, whenever $|a|\leq\dfrac{p}{p+2}$. A simple computation shows that
\begin{eqnarray}
&&\mathcal{J}_{1/2}^p(-1,1)f(z)=\frac{z\varphi'(z)}{p+\varphi(z)},\label{zvpp}\\
&&\mathcal{J}_{1/2}^p(1,-1)f(z)=-\frac{z\varphi'(z)}{p+\varphi(z)},\label{minzvpp}
\end{eqnarray}
where $\varphi(z)=\dfrac{a(p+2)z}{p+1+a(p+2)z}$. On the other hand, we have
\begin{eqnarray}\label{eqnavp}
&&\frac{z\varphi'(z)}{p+\varphi(z)}=\frac{\frac{a(p+2)z}{p}}{1+\frac{a(p+2)z}{p}}-
\frac{\frac{a(p+2)z}{p+1}}{1+\frac{a(p+2)z}{p+1}}=H(\tau)-H(\sigma),\\
&&\tau=\frac{a(p+2)z}{p},\;\sigma=\frac{a(p+2)z}{p+1},\nonumber
\end{eqnarray}
where $H(\zeta)=\dfrac{\zeta}{1+\zeta}$. Since $|\tau|<\dfrac{|a|(p+2)}{p}$, $|\sigma|<\dfrac{|a|(p+2)}{p+1}$ and $\dfrac{|a|(p+2)}{p+1}<\dfrac{|a|(p+2)}{p}$, from \eqref{eqnavp} combined with \eqref{ineqReH} we deduce the estimates
$$\operatorname{Re}\frac{z\varphi'(z)}{p+\varphi(z)}<
\sup_{|\tau|<\frac{|a|(p+2)}{p}}\operatorname{Re}H(\tau)-
\inf_{|\tau|<\frac{|a|(p+2)}{p}}\operatorname{Re}H(\tau)=\frac{2x}{1-x^2}$$
and
$$\operatorname{Re}\frac{z\varphi'(z)}{p+\varphi(z)}>
\inf_{|\tau|<\frac{|a|(p+2)}{p}}\operatorname{Re}H(\tau)-
\sup_{|\tau|<\frac{|a|(p+2)}{p}}\operatorname{Re}H(\tau)=-\frac{2x}{1-x^2},$$
where $x=\dfrac{|a|(p+2)}{p}$, i.e.
$$-\frac{2x}{1-x^2}<\operatorname{Re}\frac{z\varphi'(z)}{p+\varphi(z)}<\frac{2x}{1-x^2},
\;x=\dfrac{|a|(p+2)}{p}.$$
From the above computations, we mention that this bounds are not necessary the best ones.

Using the relations \eqref{zvpp} and \eqref{minzvpp} and according to the above inequalities, we deduce the following special cases of Corollaries \ref{cor11} -- \ref{cor14} respectively:

\begin{example}\label{ex311}
Let $a\in\mathbb{C}^*$ with
$$|a|\leq\dfrac{p}{p+2}\dfrac{-(M+p)+\sqrt{(M+p)^2+M^2}}{M},$$
where $M\geq p$, $p\in\mathbb{N}$. Then,
$$\left|\dfrac{a(p+2)z}{p+1+a(p+2)z}\right|<M,\;z\in\mathrm{U}.$$
\end{example}

\begin{example}\label{ex312}
Let $a\in\mathbb{C}^*$ with
\begin{align*}
&|a|\leq\frac{p}{p+2}\frac{-2(p-\delta)+\sqrt{4(p-\delta)^2+\delta^2}}{\delta},\;if
\;\delta\in\left[0,\frac{p}{2}\right],\\
&|a|\leq\frac{p}{p+2}\frac{-2\delta+\sqrt{(p-\delta)^2+4\delta^2}}{p-\delta},\;if\;
\delta\in\left[\frac{p}{2},p\right),
\end{align*}
where $\delta\in\left[0,p\right)$, $p\in\mathbb{N}$. Then,
$$\operatorname{Re}\left[p+\dfrac{a(p+2)z}{p+1+a(p+2)z}\right]>\delta,\;z\in\mathrm{U}.$$
\end{example}

\begin{example}\label{ex313}
Let $a\in\mathbb{C}^*$ with
$$|a|\leq\dfrac{p}{p+2}\dfrac{-(Mp+1)+\sqrt{(Mp+1)^2+M^2p^2}}{Mp},$$
where $M\geq\dfrac{1}{p}$, $p\in\mathbb{N}$. Then,
$$\left|\frac{1}{p+\frac{a(p+2)z}{p+1+a(p+2)z}}-\frac{1}{p}\right|<M,\;z\in\mathrm{U}.$$
\end{example}

\begin{example}\label{ex314}
Let $a\in\mathbb{C}^*$ with
\begin{align*}
&|a|\leq\frac{p}{p+2}\frac{-2(1-\delta p)+\sqrt{4(1-\delta p)^2+\delta^2p^2}}{\delta p},\;if
\;\delta\in\left[0,\frac{1}{2p}\right],\\
&|a|\leq\frac{p}{p+2}\frac{-2\delta p+\sqrt{(1-\delta p)^2+4\delta^2 p^2}}{1-\delta p},\;if\;
\delta\in\left[\frac{1}{2p},\frac{1}{p}\right),
\end{align*}
where $\delta\in\left[0,\dfrac{1}{p}\right)$, $p\in\mathbb{N}$. Then,
$$\operatorname{Re}\frac{1}{p+\dfrac{a(p+2)z}{p+1+a(p+2)z}}>\delta,\;z\in\mathrm{U}.$$
\end{example}

\textbf{Acknowledgments.} The first author (Dr. Pranay Goswami) is grateful to Prof. S. P. Goyal, Emeritus Scientist, Department of Mathematics, University of Rajasthan, Jaipur, India, for his guidance and constant encouragement.

The authors are grateful to the reviewers of this article, that gave valuable advices in order to revise and complete the results of the paper.



\begin{thebibliography}{99}

\bibitem{irmak}
H. Irmak and R. K. Raina, The starlikeness and convexity of multivalent functions involving certain inequalities, \textit{Rev. Math. Complut.}, \textbf{16}(2)(2003), 391--398.

\bibitem{irmak1}
H. Irmak, T. Bulboac\u{a} and N. Tuneski, Some relations between certain classes consisting of $\alpha$--convex type and Bazilevi\'{c} type functions, \textit{Appl. Math. Lett.}, \textbf{24}(12)(2011), 2010--2014.

\bibitem{miller}
S. S. Miller and P. T. Mocanu, \textit{Differential Subordinations. Theory and Applications}, Series on Monographs and Textbooks in Pure and Applied Mathematics, Vol. 225, Marcel Dekker Inc., New York and Basel, 2000.

\end{thebibliography}
\end{document}